\documentclass[a4paper,12pt]{scrartcl}

\usepackage[utf8]{inputenc}
\usepackage{amsfonts,amssymb}
\usepackage{amsmath}
\usepackage{graphicx}
\usepackage{stmaryrd}
\usepackage{bbm}
\usepackage{color}

\usepackage{a4wide}

\usepackage{empheq}
\numberwithin{equation}{section}
\usepackage{amsthm}
\theoremstyle{plain}
\newtheorem{thm}{Theorem}[section]
\newtheorem*{mainthm}{Main Theorem}
\newtheorem*{mainthmprec}{Main Theorem (precise version)}
\newtheorem{lem}[thm]{Lemma}
\newtheorem{cor}[thm]{Corollary}
\newtheorem{prop}[thm]{Proposition}
\newtheorem{obs}[thm]{Observation}

\theoremstyle{definition}

\newtheorem{exmpl}[thm]{Example}

\theoremstyle{remark}

\newtheorem*{remnns}{Remarks}

\newcommand{\R}{\mathbb{R}}
\newcommand{\Q}{\mathbb{Q}}
\newcommand{\N}{\mathbb{N}}
\newcommand{\Z}{\mathbb{Z}}

\usepackage{ifthen}
\usepackage{xargs}
\newcommand{\optionalsubsuper}[3]{%
\ifthenelse{\equal{#2}{}}{%
  \ifthenelse{\equal{#3}{}}{%
    #1}{%
    #1^{#3}}
  }{%
  \ifthenelse{\equal{#3}{}}{%
    #1_{#2}}{%
    #1_{#2}^{#3}}}}

\newcommand{\union}{\cup}

\newcommand{\intersect}{\cap}

\renewcommand{\implies}{\Rightarrow}
\newcommand{\into}{\hookrightarrow}

\newcommand{\pow}{\mathfrak{P}}
\newcommand{\defeq}{\mathrel{\mathop{:}}=}

\newcommand{\scaprod}[2]{\left({}#1\mid{}#2\right)}

\newcommand{\gen}[1]{\left\langle #1\right\rangle}
\newcommand{\abs}[1]{\left\lvert#1\right\rvert}

\newcommand{\Family}{\mathfrak{F}}
\newcommandx{\OrbitCat}[2][1=\Group,2=\Family]%
   {\optionalsubsuper{\mathcal{O}}{#2}{}#1}

\newcommand{\Skeleton}[2]{#1^{(#2)}}

\newcommand{\Group}{G}

\DeclareMathOperator{\ind}{ind}
\DeclareMathOperator{\length}{length}

\newcommandx{\BredonFP}[2][1=\Family]{\ensuremath{#1\textrm{-}\mathrm{FP}_{#2}}}

\DeclareMathOperator{\diag}{diag}

\DeclareMathOperator{\SL}{SL}
\DeclareMathOperator{\GL}{GL}

\DeclareMathOperator{\CAT}{CAT}

\newcommand{\AbelsScheme}{\mathbf{G}}
\newcommand{\BorelScheme}{\mathbf{B}}
\newcommand{\TorusScheme}{\mathbf{T}}
\newcommand{\UniScheme}{\mathbf{U}}

\newcommand{\Space}{X}
\newcommand{\Building}{X}
\newcommand{\EBuilding}{\Building^1}
\newcommand{\SimpBuilding}{\Delta}
\newcommand{\ESimpBuilding}{\SimpBuilding^1}
\newcommand{\EApartment}{\Sigma^1}
\newcommand{\Apartment}{\Sigma}

\newcommand{\ECone}{\bar{C}^1}

\newcommand{\EChamber}{C^1}
\newcommand{\InnerChamber}{c}
\newcommand{\AltInnerChamber}{d}
\newcommand{\DGroup}{\Gamma}
\newcommand{\Abels}{G}
\newcommand{\Borel}{B}
\newcommand{\Torus}{T}
\newcommand{\Uni}{U}

\newcommand{\AltVector}{w}
\newcommand{\Point}{x}
\newcommand{\Vector}{v}
\newcommand{\VecOne}{\Vector^1}
\newcommand{\VecTwo}{\Vector^2}
\newcommand{\VecNum}[1]{\Vector^{#1}}
\newcommand{\VecOT}{\Vector}
\newcommand{\PointAtInfinity}{\xi}
\newcommand{\AltPointAtInfinity}{\zeta}
\newcommand{\Busemann}{\beta}
\newcommand{\BuseOne}{\Busemann^1}
\newcommand{\BuseTwo}{\Busemann^2}
\newcommand{\BuseNum}[1]{\Busemann^{#1}}
\newcommand{\BuseOT}{\Busemann}
\newcommand{\PAIOne}{\PointAtInfinity^1}
\newcommand{\PAITwo}{\PointAtInfinity^2}
\newcommand{\PAINum}[1]{\PointAtInfinity^{#1}}
\newcommand{\PAIOT}{\PointAtInfinity}
\newcommand{\HyperOne}{H^1}
\newcommand{\HyperTwo}{H^2}
\newcommand{\HyperOT}{H}
\newcommand{\VecPerp}{\ell}
\newcommand{\Project}{\operatorname{pr}_1}
\newcommand{\Ray}{\gamma}

\newcommand{\LocalField}{K}
\newcommand{\Ring}{\mathcal{O}}

\newcommand{\Valuation}{\nu}
\newcommand{\Uniformizer}{\pi}

\newcommand{\Lattice}{\Lambda}
\newcommand{\LCVector}{f}
\newcommand{\EqClass}[1]{[#1]}

\newcommand{\Root}{\alpha}

\newcommand{\BredonF}{\underbar{\ensuremath{\mathit{F}}\!}\,}
\newcommand{\BredonFPfin}{\underbar{\ensuremath{\mathit{FP}}\!}\,}
\newcommand{\ClassicalF}{\mathit{F}}
\newcommand{\ClassicalFP}{\mathit{FP}}

\newcommand{\Norm}{\alpha}

\newcommand{\ed}{\operatorname{ed}}

\newcommand{\Retraction}{\rho}
\newcommand{\GroupRetraction}{\eta}

\newcommand{\Lk}{\operatorname{Lk}}

\newcommand{\idmatrix}{\operatorname{id}}
\newcommand{\Involution}{\sigma}
\newcommand{\FixedPointSet}{Y}
\newcommand{\FiltrationSet}{Z}
\newcommand{\ExtendingFactor}{L}
\newcommand{\Angle}{\theta}
\newcommand{\SimpLine}{\R}
\newcommand{\ABasis}{a}
\newcommand{\Basis}{e}

\newcommand{\Partition}{\mathcal{I}}

\newcommand{\trealize}[1]{{\lvert {#1} \rvert}}
\newcommand{\omod}{\mathop{\mathrm{mod}}}
\newcommand{\odiv}{\mathop{\mathrm{div}}}

\title{Abels's groups revisited}
\author{Stefan Witzel}

\begin{document}

\maketitle

\begin{abstract}
\noindent
We generalize a class of groups introduced by Herbert Abels to produce examples of virtually torsion free groups that
have Bredon-finiteness length $m-1$ and classical finiteness length $n-1$ for all $0 < m \le n$.

\noindent
The proof illustrates how Bredon-finiteness properties can be verified using geometric methods and a version of Brown's criterion due to Martin Fluch and the author.
\end{abstract}

Let $\AbelsScheme_n$ be the algebraic group of invertible upper triangular $(n+1)$-by-$(n+1)$ matrices whose extremal diagonal entries are $1$. The groups $\AbelsScheme_n(\Z[1/p])$ where $p$ is a prime were introduced by Abels because they have interesting finiteness properties. Namely, it was shown in \cite{abels79,strebel84,abebro87,brown87} that $\AbelsScheme_n(\Z[1/p])$ is of type $\ClassicalF_{n-1}$ but not of type $\ClassicalF_n$.

Recall that a group $\DGroup$ is \emph{of type $\ClassicalF_n$} if it admits a classifying space $X$ whose $n$-skeleton $\Skeleton{X}{n}$ is compact modulo the action of $\DGroup$.
A \emph{classifying space} is a contractible CW complex on which $\DGroup$ acts freely. Closely related to these topological finiteness properties are the homological finiteness properties of being of type $\ClassicalFP_n$: by definition $\DGroup$ is \emph{of type $\ClassicalFP_n$} if the trivial $\Z\DGroup$-module $\Z$ admits a projective resolution $(P_i)_{i \in \N}$ with $P_i$ finitely generated for $i \le n$. It is not hard to see that $\ClassicalF_n \implies \ClassicalFP_n$.

For a group $\DGroup$ that has torsion it is sometimes more natural to consider a \emph{classifying space for proper
actions} for $\DGroup$. This is a CW complex on which $\DGroup$ acts rigidly in such a way that the fixed point set of
every finite subgroup is (nonempty and) contractible and the fixed point set of every infinite subgroup is empty. We say
that $\DGroup$ is \emph{of type $\BredonF_n$} if it admits a classifying space for proper actions whose $n$-skeleton is
compact modulo the action of $\DGroup$. There is a homology theory developed by Glen Bredon \cite{bredon67} and
generalized by Wolfgang Lück \cite{lueck89} that describes the homological aspects of proper actions just as
usual homology does for free actions. In particular, we get a notion of Bredon-finiteness properties $\BredonFPfin_n$
and again $\BredonF_n \implies \BredonFPfin_n$. For the definition we refer the reader to \cite{fluwit}.

The lower finiteness properties have more concrete interpretations: a group is of type $\ClassicalF_1$ if and only if it
is finitely generated, it is of type $\ClassicalF_2$ if and only if it is finitely presented, and it is of type
$\BredonFPfin_0$ if and only if it has finitely many conjugacy classes of finite subgroups,
\cite[Lemma~3.1]{kromapnuc09}.

We define the \emph{classical finiteness length} of $\DGroup$ to be the supremum over all those $n$ for which $\DGroup$
is of type $\ClassicalFP_n$. The \emph{Bredon-finiteness length} is defined analogously. We can now state a version of
our Main~Theorem.

\begin{mainthm}
For $0 < m \le n$ there is a solvable algebraic group $\AbelsScheme$ such that for every odd prime $p$ the group $\AbelsScheme(\Z[1/p])$ has classical finiteness length $n-1$ and has Bredon-finiteness length $m-1$.
\end{mainthm}

Related examples were obtained via more algebraic means in \cite{kocmarnuc11}. Other examples of groups with
torsion that separate between Bredon-finiteness properties can be found in
\cite[Examples~3,4]{leanuc03}.

The precise version of the Main~Theorem depends on some combinatorial conditions which are formulated in Section~\ref{sec:precise}. After some basic facts about Bruhat--Tits buildings and $\CAT(0)$-spaces in Section~\ref{sec:buildings}, we establish the classical finiteness length of the groups in Section~\ref{sec:classical}. The Bredon-finiteness length is verified in Section~\ref{sec:bredon}. Appendix~\ref{sec:extended_simplicial} describes the natural simplicial model for the extended Bruhat--Tits building of $\GL_n(\LocalField)$. This should be well known but the author could not find a good reference.

\paragraph{Acknowledgments.}I would like to thank Herbert Abels, Kai-Uwe Bux, Martin Fluch, Giovanni Gandini, Linus
Kramer, Ian Leary, Marco Schwandt, Daniel Skodlerack and Matthew Zaremsky for helpful comments. I also
greatfully acknowledge support through the SFBs 878 in Münster and 701 in Bielefeld.

\section{Precise statement of Main Theorem}
\label{sec:precise}

Fix $n \in \N$ and consider two nonzero integer vectors
\[
\VecOne = (\VecOne_1,\ldots,\VecOne_{n+1}) \quad \text{and}\quad \VecTwo = (\VecTwo_1,\ldots,\VecTwo_{n+1})
\]
which satisfy the following conditions:
\begin{enumerate}
\item The sequences $(\VecOne_i)_i$ and $(\VecTwo_i)_i$ are monotonically decreasing.\label{item:monotonicity}
\item $\sum_{i} \VecOne_i > 0$ and $\sum_i \VecTwo_i \le 0$.\label{item:position_relative_sigma}
\end{enumerate}
Denote by $\AbelsScheme_{\VecOne,\VecTwo}$ the algebraic group defined by
\[
\AbelsScheme_{\VecOne,\VecTwo}(A) = \left\{
\left(
\begin{array}{cccc}
d_1& * & \cdots & *\\
0 & \ddots & \ddots & \vdots \\
\vdots & \ddots & \ddots & * \\
0 & \cdots & 0 & d_{n+1}
\end{array}
\right)
\in \GL_{n+1}(A)
\mathrel{\Bigg|}
\prod_i d_i^{\VecOne_i} = 1 = \prod_i d_i^{\VecTwo_i}
\right\}
\]

We define a new vector by
\[
\VecOT \defeq \VecTwo - \frac{\sum_i \VecTwo_i}{\sum_i \VecOne_i} \VecOne\text{ .}
\]
Note that $\VecOT$ satisfies $\sum_i \VecOT_i = 0$.

By a \emph{partition} of $I \defeq \{1,\ldots,n+1\}$ we mean a set $\Partition \subseteq \pow_{\ne \emptyset}(I)$ of nonempty subsets, its \emph{blocks}, that disjointly cover $I$.
A partition $\{J^+,J^-\}$ is called \emph{elementary admissible} (relative to $\VecOne$ and $\VecTwo$) if $\sum_{i \in J^-} \VecOne_i$  and $\sum_{i \in J^-} \VecTwo_i$ are even. The trivial partition $\{I\}$ is also considered elementary admissible. A partition is called \emph{admissible} if it is the (coarsest) common refinement of elementary admissible partitions.
We say that $\Partition$ is a \emph{partition of $\VecOT$} if $\sum_{i \in J} \VecOT_i =0$ for every block $J$ of
$\Partition$. The \emph{essential blocks} of a partition of $\VecOT$ are the blocks $J$ on which $\VecOT$ is not
constant zero, that is, for which there is an $i \in J$ such that $\VecOT_i \ne 0$. The \emph{essential dimension} of a
partition of $\VecOT$ is
\[
\ed(\Partition) = \sum_{J} (\abs{J} - 1)
\]
where the sum runs over the essential blocks of $\Partition$. We can now define
\[
m = m(\VecOne,\VecTwo) \defeq \min \{ \ed(\Partition) \mid \Partition \text{ is an admissible partition of }\VecOT\}
\]
and state:

\begin{mainthmprec}
Let $n$, $\VecOne$, $\VecTwo$, and $m$ be as above. For every odd prime $p$ the group $\AbelsScheme_{\VecOne,\VecTwo}(\Z[1/p])$ is of type $\ClassicalF_{n-1}$ but not of type $\ClassicalFP_{n}$ and is of type $\BredonFPfin_{m-1}$ but not of type $\BredonFPfin_{m}$.
\end{mainthmprec}

\begin{remnns}
\begin{enumerate}
\item Since the trivial partition is admissible, we have $m \le n$. Since every essential block of a partition of $\VecOT$ must have size at least $2$, we have $m \ge 1$.
\item Admissibility of a partition is not a strong restriction. In fact, if all entries of $\VecOne$ and $\VecTwo$ are even, then every partition is admissible.
\item That the Main Theorem only shows the group to be of type $\BredonFPfin_{m-1}$ instead of $\BredonF_{m-1}$ is due to the fact that there is no version of Brown's criterion for $\BredonF_2$. The reason is that \cite{brown84} does not directly translate to the context of proper actions. Once a criterion for $\BredonF_2$ is available, our method of proof should give type $\BredonF_{m-1}$.
\item The restriction to odd primes is due to the fact that involutions in the building associated to $\GL_{n+1}(\Q_2)$ have larger fixed point set than they should, cf.\ Proposition~\ref{prop:involution_fixed}. In the case $p=2$ Lemmas~\ref{lem:only_2-torsion} and \ref{lem:conjugate_to_diagonal} imply that the group is of type $\BredonFPfin_0$. It is not clear to the author what the higher Bredon-finiteness properties are in that case.
\end{enumerate}
\end{remnns}

We give some examples which in particular allow us to recover the previous formulation of the Main~Theorem. Denote the standard basis of $\Z^{n+1}$ by $\ABasis_1,\ldots,\ABasis_{n+1}$.

\begin{exmpl}
If $\VecOne = \ABasis_1$ and $\VecTwo = -\ABasis_{n+1}$, then $\Gamma = \AbelsScheme_{\VecOne,\VecTwo}(\Z[1/p])$ is just Abels's group $\AbelsScheme_n(\Z[1/p])$. In this case $\VecOT = \ABasis_1 - \ABasis_{n+1}$ and the elementary admissible partition into $J^+ = \{1,n+1\}$ and $J^- = \{2,\ldots,n\}$ shows that $m=1$. Therefore, the Main~Theorem states that $\Gamma$ is of type $\ClassicalF_{n-1}$ but not of type $\ClassicalFP_n$ and is of type $\BredonFPfin_0$ but not of type $\BredonFPfin_1$. The classical finiteness length was known by \cite[Theorem~A]{abebro87} and \cite[Theorem~6.1]{brown87}. To prove the first part of the theorem we use metric versions of the methods used there. Part of the translation is done in Appendix~\ref{sec:extended_simplicial}.
\end{exmpl}

\begin{exmpl}
For $0 < m \le n$, we may take $\VecOne = 2\sum_i \ABasis_i$ and $\VecTwo = -m \ABasis_{n+1} + \sum_{i = 1}^m \ABasis_i$.
Then $\VecOT = \VecTwo$ and every partition of $\VecOT$ must contain $\{1,\ldots,m,n+1\}$ in one block and therefore have essential dimension at least $m$. The partition into $J^+ = \{1,\ldots,m,n+1\}$ and $J^- = \{m+1,\ldots,n\}$ is elementary admissible and has essential dimension $m$. Thus we get groups of Bredon-finiteness length $m-1$ and classical finiteness length $n-1$ and recover the original formulation of the Main~Theorem.
\end{exmpl}

\begin{exmpl}
As an example of how admissibility comes into play let $n = 2k$ be even and consider the vectors $\VecOne = \ABasis_1 + \ldots + \ABasis_{k+1}$ and $\VecTwo = - \ABasis_{k+1} - \ldots - \ABasis_{n+1}$. Then $\VecOT = \ABasis_1 + \ldots + \ABasis_k - \ABasis_{k+2} - \ldots - \ABasis_{n+1}$. A partition of $\VecOT$ with the minimal essential dimension of $k$ is into $\{1,n+1\}, \ldots, \{k,k+2\}$. However, this partition is not admissible. If $k$ is even, a partition of $\VecOT$ of the minimal admissible essential dimension of $3/2 \cdot k$ is into $\{1,2,n,n+1\}, \ldots, \{k-1,k,k+2,k+3\}$. If $k$ is odd, the minimal admissible essential dimension is $3/2 \cdot (k-1) + 2$ and realized by the partition $\{1,2,n,n+1\}, \ldots, \{k-2,k-1,k+3,k+4\}, \{k,k+1,k+2\}$.
So if we set $\Gamma = \AbelsScheme_{\VecOne,\VecTwo}(\Z[1/p])$ and $\Gamma' = \AbelsScheme_{2\VecOne,2\VecTwo}(\Z[1/p])$, we get: $\Gamma$ is a subgroup of finite index in $\Gamma'$. The Bredon-finiteness length of $\Gamma'$ is $k$ while the Bredon-finiteness length of $\Gamma$ is $3/2\cdot k$ respectively $3/2\cdot (k-1) + 2$. Of course, all groups considered here are virtually torsion free and hence virtually of type $\BredonFPfin_{n-1}$.
\end{exmpl}

The plan to prove the Main~Theorem is as follows. The group $\DGroup \defeq \AbelsScheme(\Z[1/p])$ acts on the extended Bruhat--Tits building $\EBuilding$ associated to $\GL_{n+1}(\Q_p)$. Cell stabilizers are arithmetic and thus of type $\ClassicalF_\infty$. The vectors $\VecOne$ and $\VecTwo$ define horospheres $\HyperOne$ and $\HyperTwo$ that are invariant under the action of $\DGroup$. Moreover, the action of $\DGroup$ on $\HyperOne \intersect \HyperTwo$ is cocompact. The horosphere $\HyperOne$ can be identified with the (non-extended) Bruhat--Tits building $\Building$ in such a way that $\HyperOne \intersect \HyperTwo$ is identified with a horosphere in $\Building$. It is known that horospheres in $\Building$ are $(n-2)$-connected. More precisely, let $\Busemann$ be the Busemann function whose $0$-level is the horosphere. Then the maps $\Busemann^{-1}([0,s]) \into \Busemann^{-1}([0,s+1])$ induce isomorphisms in $\pi_{k}$ for $k<n-1$ and epimorphisms that are infinitely often non-injective in $\pi_{n-1}$.

With these ingredients, the classical finiteness length follows from Brown's criterion, which we state below. But first we have to recall some definitions. Recall that a space $\Space$ is \emph{$n$-connected} if $\pi_k(\Space) = 1$ for $k \le n$ and is \emph{$n$-acyclic} if $\tilde{H}_k(\Space) = 0$ for $k \le n$. The action of a group $\DGroup$ on a CW-complex $\FiltrationSet$ is called \emph{rigid} if the stabilizer of every cell fixes that cell pointwise. A system of groups $(A_s \to A_{s+1})_{s \in \N}$ is called \emph{essentially trivial} if for every $s$ there is a $t \ge s$ such that the map $A_s \to A_{t}$ is trivial.

\begin{thm}[Brown's criterion {\cite[Theorems~2.2,3.2]{brown87}}]
\label{thm:browns_criterion}
Let $\DGroup$ act rigidly on an CW-complex $\FiltrationSet$. Assume that $\FiltrationSet$ is $(n-1)$-connected. Assume also that the stabilizer of each $k$-cell is of type $\ClassicalF_{n-k}$. Let $(\FiltrationSet_s)_{s \in \N}$ be a filtration of $\FiltrationSet$ by $\DGroup$-invariant and $\DGroup$-cocompact subspaces. Then $\DGroup$ is of type $\ClassicalF_n$ if and only if the system
\[
\pi_k(\FiltrationSet_s) \to \pi_k(\FiltrationSet_{s+1})
\]
is essentially trivial for $k<n$.
The same statement holds with ``$(n-1)$-connected'' replaced by ``$(n-1)$-acyclic'', ``$\pi_k$'' replaced by ``$\tilde{H}_k$'', and ``$\ClassicalF_n$'' replaced by ``$\ClassicalFP_n$''.
\end{thm}

To determine the Bredon-finiteness length, we have to take torsion into account. The only torsion elements that $\DGroup$ contains are of order $2$. Moreover, every finite subgroup is conjugate to a group of diagonal matrices. The fixed point set of such a group is a product of extended Bruhat--Tits buildings. More precisely, it is the extended Bruhat--Tits building of the centralizer of the finite group. The products that can arise are described by admissible partitions. The horosphere in the fixed point set is a product of a horosphere in the essential factors, those corresponding to essential blocks, and of the remaining factors. Its connectivity is two less than the essential dimension. From this, the Bredon-finiteness length can be deduced using the following version of Brown's criterion from \cite{fluwit}:

\begin{thm}
\label{thm:browns_criterion_for_bredon}
Let $\DGroup$ act rigidly on a CW-complex $\FiltrationSet$. Assume that for every finite subgroup $F < \DGroup$ the fixed point set $\FiltrationSet^F$ is $(n-1)$-acyclic. Assume also that the stabilizer of each $k$-cell is of type $\BredonFPfin_{n-k}$. Let $(\FiltrationSet_s)_{s \in \N}$ be a filtration of $\FiltrationSet$ by $\DGroup$-invariant and $\DGroup$-cocompact subspaces. Then $\DGroup$ is of type $\BredonFPfin_n$ if and only if for $k<n$ the following holds: for every $s$ there is an $t \ge s$ such that the maps
\[
\tilde{H}_k(\FiltrationSet_s^F) \to \tilde{H}_k(\FiltrationSet_{t}^F)
\]
are trivial for all finite subgroups $F$.
\end{thm}

\section{Buildings}
\label{sec:buildings}

From now on we fix $n$, $\VecOne$, $\VecTwo$, and $p$ and write $\AbelsScheme$ for $\AbelsScheme_{\VecOne,\VecTwo}$ and put $\DGroup \defeq \AbelsScheme(\Z[1/p])$. To prove the theorem we have to exhibit a space $\FiltrationSet_0$ on which $\DGroup$ acts cocompactly with good stabilizers. The finiteness properties of $\DGroup$ will then correspond to the connectivity of $\FiltrationSet_0$.

The starting point for the construction of $\FiltrationSet_0$ is the Bruhat--Tits building $\Building$ associated to $\GL_{n+1}(\Q_p)$. Recall that $\Building$ is a thick, irreducible, euclidean building of type $\tilde{A}_n$ and in particular is a $\CAT(0)$-space \cite[Theorem~11.16]{abrbro}. We denote by $\EBuilding$ the extended building $\Building \times \ExtendingFactor$, where $\ExtendingFactor$ is a euclidean line. The action of $\GL_{n+1}(\Q_p)$ on $\EBuilding$ is given by
\begin{equation}
\label{eq:extended_action}
g.(x,r) = (g.x,r-\frac{1}{n+1}\Valuation(\det g)) \text{ ,}
\end{equation}
see \cite[Paragraphe~2]{brutit84b}. We write $\Project\colon\EBuilding \to \Building$ for the projection onto the first factor.

We will consider the following subgroups of $\GL_{n+1}$: the group $\BorelScheme$ of upper triangular matrices, the group $\TorusScheme$ of diagonal matrices, and the group $\UniScheme$ of strict upper triangular matrices. Non-bold letters will denote the corresponding groups of $\Q_p$ points, that is $\Abels = \AbelsScheme(\Q_p)$, $\Borel = \BorelScheme(\Q_p)$ and so on.

There is a unique apartment $\EApartment$ of $\EBuilding$ that is invariant under the action of $\Torus$. We can identify $\EApartment$ with $\R^{n+1}$ in such a way that the action of $\Torus$ is given by
\begin{equation}
\label{eq:torus_action}
\diag(d_1,\ldots,d_{n+1}).(\Point_1,\ldots,\Point_{n+1}) = (\Point_1 + \Valuation(d_1),\ldots,\Point_{n+1} + \Valuation(d_{n+1})) \text{ .}
\end{equation}
With this identification the apartment $\Apartment \defeq \EApartment \intersect \Building$ of $\Building$ is the hyperplane $\VecPerp^\perp$ where $\VecPerp = (1,\ldots,1)$.
The boundary $\partial \EApartment$ is an apartment of the spherical building $\partial \EBuilding$ that is fixed by $\Torus$. The group $\Borel$ fixes a chamber $\EChamber$ of $\partial \EApartment$. Making use of the above identification, we can describe the chamber $\EChamber$ as follows. The standard root system $\Root_i = \ABasis_i - \ABasis_{i+1}$, $1 \le i \le n$ of type $A_n$ defines a cone
\[
\ECone = \{\Point \in \EApartment \mid \scaprod{\Root_i}{x} \ge 0\}
\]
in $\EApartment$ and $\EChamber$ is the boundary of $\ECone$.

As a last ingredient from the theory of buildings consider the morphism $\GroupRetraction \colon \Borel \to \Torus$ that takes each matrix to its diagonal. Its kernel is $\Uni$. There is a corresponding map $\Retraction \colon \EBuilding \to \EApartment$, the retraction onto $\EApartment$ centered at $\EChamber$. It can be described by the property that it takes a ray $\Ray \colon [0,\infty) \to \EBuilding$ whose endpoint lies in $\EChamber$ to a ray $\Retraction \circ \Ray \colon [0,\infty) \to \EApartment$ that coincides with the original ray on an infinite interval. Both maps are linked by the relation
\begin{equation}
\label{eq:retraction_relation}
\Retraction(b.\Point) = \GroupRetraction(b).\Retraction(\Point)
\end{equation}
for $b \in \Borel$ and $\Point \in \EBuilding$. The image under $\GroupRetraction$ of $\Abels$ is just $\Abels \intersect \Torus$ (and similarly for $\DGroup$).

As a consequence of \eqref{eq:retraction_relation} we observe that $\Uni$ not only fixes $\EChamber$ but for every point $\PointAtInfinity \in \EChamber$ leaves invariant every Busemann function centered at $\PointAtInfinity$.
%If a group or a group element has this property, then we say that it \emph{horofixes} $\PointAtInfinity$.

\section{Classical finiteness properties}
\label{sec:classical}

It is time to shed some light on the seemingly mysterious notions of Section~\ref{sec:precise}. Our group $\DGroup$ is a subgroup of $\GL_{n+1}(\Q_p)$ and therefore acts on the extended building $\EBuilding$. From \eqref{eq:torus_action} we see:

\begin{obs}
Let $\AltVector$ be a vector in $\EApartment$. An element $g = \diag(d_i) \in T$ leaves $\AltVector^\perp$ invariant if and only if $\sum_{i} \AltVector_i \Valuation(d_i) = 0$.
\end{obs}

Assume in addition that $\AltVector \in \Z^{n+1}$ and write $d_i = a_i \cdot \Uniformizer^{\Valuation(d_i)}$ with $a_i \in \Ring^\times$. Then
$\prod_i d_i^{\AltVector_i} = \Uniformizer^{\sum_{i} \AltVector_i \Valuation(d_i)} \cdot \prod_i a_i^{\AltVector_i}$
where the second factor lies in $\Ring^\times$. Therefore $\prod_i d_i^{\AltVector_i} =1$
is sufficient for $g$ to leave $\AltVector^\perp$ invariant. Regarding $\VecOne$ and $\VecTwo$ as vectors in $\EApartment$ we obtain with \eqref{eq:retraction_relation}:

\begin{obs}
The group $\Abels \intersect \Torus$ leaves $(\VecOne)^\perp$ and $(\VecTwo)^\perp$ invariant. Consequently $\Abels$ leaves $\HyperOne \defeq \Retraction^{-1}((\VecOne)^\perp)$ and $\HyperTwo \defeq \Retraction^{-1}((\VecTwo)^\perp)$ invariant.
\end{obs}

This discussion suggests that $\HyperOne \intersect \HyperTwo$ is the right space for $\DGroup$ to act on. Condition \eqref{item:monotonicity} means that $\VecOne$ and $\VecTwo$ point into $\ECone$. This in turn implies that $\HyperOne$ and $\HyperTwo$ are in fact horospheres. Indeed, let $\PAINum{j}$ be the endpoint of the geodesic ray spanned by $\VecNum{j}$ and let $\BuseNum{j}$ be the Busemann function corresponding to $[0,\PAINum{j})$. Then $\BuseNum{j}((\VecNum{j})^\perp) = 0$ and the fact that $\PAINum{j} \in \EChamber$ implies that $\BuseNum{j} = \BuseNum{j} \circ \Retraction$.

Condition \eqref{item:position_relative_sigma} implies that $\VecOne$ does not lie in $\Apartment$ but the geodesic segment $[\VecOne,\VecTwo]$ meets $\Apartment$. In fact, the intersection point is just $\VecOT$. As before, let $\PAIOT$ be the endpoint of the geodesic ray spanned by $\VecOT$ and let $\BuseOT$ be the corresponding Busemann function.

\begin{lem}
\label{lem:identification}
The restriction $\Project|_{\HyperOne}$ is a homeomorphism that takes horoballs centered at $\PAITwo$ to horoballs centered at $\PAIOT$.
\end{lem}

We prove a more general statement:

\begin{prop}
\label{prop:generalized_identification}
Let $\EBuilding$ be a $\CAT(0)$-space that decomposes as $\EBuilding = \Building \times \ExtendingFactor$ where $\ExtendingFactor$ is a euclidean line. Let $\Project \colon \EBuilding \to \Building$ be the projection onto the first factor. Let $\PAIOne \in \partial \EBuilding \setminus \partial \Building$ and let $\HyperOne$ be a horosphere centered at $\PAIOne$. The restriction $\Project|_{\HyperOne}$ is a homeomorphism.

Moreover, if $\PAITwo \in \partial \EBuilding$ is such that $\angle(\PAIOne,\PAITwo) \ne \pi$ and that the unique geodesic $[\PAIOne,\PAITwo]$ meets $\partial \Building$ in a point $\PAIOT$, then $\Project|_{\HyperOne}$ takes horoballs around $\PAITwo$ to horoballs around $\PAIOT$.
\end{prop}

\begin{proof}
We identify $\ExtendingFactor$ with $\R$ and write elements of $\EBuilding$ as pairs $(\Point,r)$ with $\Point \in \Building$ and $r \in \R$. We also let  $\infty$ and $-\infty$ denote the endpoints of $\ExtendingFactor$. Let $\BuseOne$ be the Busemann function centered at $\PAIOne$ so that $\HyperOne = (\BuseOne)^{-1}(0)$. For $\Point \in \Building$ we may consider the euclidean half-plane spanned by the geodesic ray $\Project([\Point,\PAIOne])$ and the line $\ExtendingFactor$. In that half-plane it is easy to verify that
\begin{equation}
\label{eq:busemann_difference}
\BuseOne(x,r) - \BuseOne(x,s) = \cos \Angle_1 (r-s)
\end{equation}
where $\Angle_1 = \angle(\PAIOne,\infty)$ (see Figure~\ref{fig:busemann}). From this it follows that $\Project|_{\HyperOne}$ is a homeomorphism with inverse
\[
x \mapsto \Big(x,-\frac{1}{\cos \Angle_1} \BuseOne(x,0)\Big) \text{ .}
\]

\begin{figure}
\begin{center}
\includegraphics{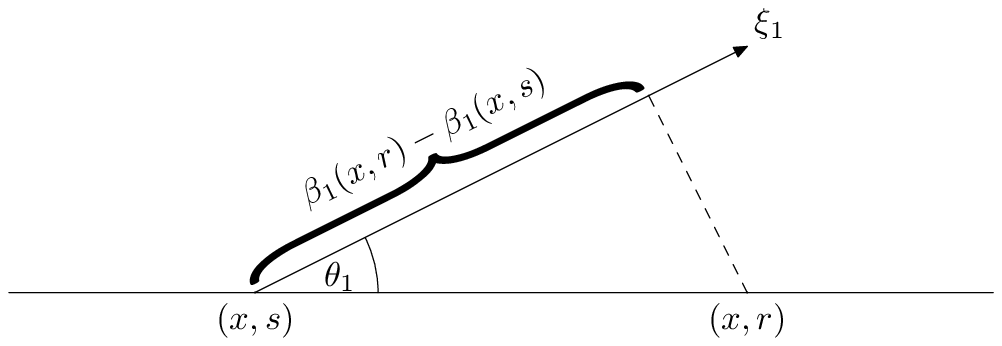}
\end{center}
\caption{$\protect\BuseOne(x,r) - \protect\BuseOne(x,s) = \cos \Angle_1 (r-s)$}
\label{fig:busemann}
\end{figure}

For the second statement set $\Angle_2 = \angle(\PAITwo,\infty)$ and observe that \eqref{eq:busemann_difference} holds
analogously. We define $\BuseOT \defeq \BuseTwo - (\cos \Angle_2)/(\cos \Angle_1) \BuseOne$. Note that this is a
positive combination of $\BuseOne$ and $\BuseTwo$ by the assumption that $[\xi_1,\xi_2] \intersect \partial \Building
\ne \emptyset$. Therefore it is up to scaling a Busemann function centered at a point in $[\PAIOne,\PAITwo]$. Moreover,
\begin{align*}
\BuseOT(x,r) - \BuseOT(x,s) & = \BuseTwo(x,r) - \BuseTwo(x,s) - \frac{\cos \Angle_2}{\cos \Angle_1} (\BuseOne(x,r) - \BuseOne(x,s))\\
& = \cos \Angle_2 (r-s) - \frac{\cos\Angle_2}{\cos \Angle_1} \cos \Angle_1 (r-s)\\
& = 0
\end{align*}
hence $\BuseOT$ is centered at $\PAIOT$ and we may in particular regard it as a reparametrized Busemann function on $\Building$. For $(x,r) \in \EBuilding$ with $\BuseOne(x,r) = 0$ we clearly have $\BuseTwo(x,r) = \BuseOT(x,r)$ which shows the second claim.
\end{proof}

%\begin{rem}
%There is a more conceptual description of this situation. Choosing a basepoint, the apartment $\EApartment$ can be regarded as the $\R$-span $(X_*(\TorusScheme))_\R$ of the free $\Z$-module of cocharacters $X_*(\TorusScheme) = \Hom(\GL_1,\TorusScheme)$ (cf.\ \cite[Définition~2.1.8]{rousseau77}). This module contains in particular the coroots $\CoRoot_i$. Let $X^*(\TorusScheme) = \Hom(\TorusScheme,\GL_1)$ be the $\Z$-module of characters. There is a dual pairing $\pairing{\cdot}{\cdot}$ of $X^*(\TorusScheme)$ and $X_*(\TorusScheme)$ (see for example \cite[Corollary~1]{borel91}) which allows us to regard $(X^*(\TorusScheme))_\R$ as the dual space of $\EApartment$.
%
%The vectors $\VecNum{j}$ define characters $\chi^j \in X^*(\TorusScheme)$. The condition that $\VecNum{j} \in \ECone$ is saying that $\pairing{\chi^j}{\CoRoot_i} \ge 0$ for every $i$. The group $\AbelsScheme$ is just $\ker \chi^1 \intersect \ker \chi^2$. With the above description, the $\chi^j$ are elements of $(\EApartment)^*$ and coincide on $\EApartment$ with $\Busemann^j$. Thus $\Busemann^j = \chi^j \circ \Retraction$.
%\end{rem}

Our next goal is to show that the action of $\DGroup$ on $\HyperOne \intersect \HyperTwo$ is cocompact. The first step is the following consequence of the cocompactness result of \cite{abebro87}.

\begin{prop}
\label{prop:borel_cover}
The building $\EBuilding$ is covered by translates of $\EApartment$ under $\BorelScheme(\Z[1/p])$. In short, $\BorelScheme(\Z[1/p]).\EApartment = \EBuilding$.
\end{prop}

\begin{proof}
By \cite[Proposition~2.1~(b)]{abebro87} $\BorelScheme(\Z[1/p])$ acts transitively on the lattices in $\Q_p^{n+1}$ which by Appendix~\ref{sec:extended_simplicial} correspond to the vertices of $\EBuilding$. Let $\InnerChamber$ be a chamber of $\EBuilding$ and let $\Ray$ be a geodesic ray that starts in a vertex, ends in $\EChamber$ and meets the interior of $\InnerChamber$ at a time $t$. Let $\Ray' = \Retraction \circ \Ray$, which is also a geodesic ray because $\Ray$ ends in $\EChamber$ and $\Retraction$ is centered at $\EChamber$. Let $\AltInnerChamber \subseteq \EApartment$ be the chamber that contains $\Ray'(t)$. Let $g \in \BorelScheme(\Z[1/p])$ be such that $g.\Ray'(0) = \Ray(0)$. Since $\Borel$ fixes $\EChamber$ it follows that  $g \circ \Ray' = \Ray$ and in particular $g.\AltInnerChamber = \InnerChamber$.
\end{proof}

Now cocompactness of $\DGroup$ follows using that $\VecOne$ and $\VecTwo$ lie in $\Z^n$.

\begin{lem}
\label{lem:cocompactness}
\begin{enumerate}
\item $\DGroup \intersect \Torus$ acts cocompactly on $(\VecOne)^\perp \intersect (\VecTwo)^\perp$.
\item $\DGroup$ acts cocompactly on $\HyperOne \intersect \HyperTwo$.
\end{enumerate}
\end{lem}

\begin{proof}
For the first part note that $\TorusScheme(\Z[1/p])$ acts on $\EApartment$ through $\Z^n$ and the intersection $\DGroup \intersect \Torus$ acts as the stabilizer in $\Z^n$ of $(\VecOne)^\perp \intersect (\VecTwo)^\perp$. Since $\VecOne$ and $\VecTwo$ lie in $\Z^n$, the $\Z$-module $(\VecOne)^\perp \intersect (\VecTwo)^\perp \intersect \Z^n$ has rank $n-2$, so the stabilizer acts cocompactly.

Now let $K \subseteq \EApartment$ be compact such that its translates under $\DGroup$ cover $(\VecOne)^\perp \intersect (\VecTwo)^\perp$ and let $\Point \in \HyperOne \intersect \HyperTwo$ be arbitrary. By Proposition~\ref{prop:borel_cover} there is a $b \in \BorelScheme(\Z[1/p])$ such that $b.x \in \EApartment$. Clearly there is an $s \in \TorusScheme(\Z[1/p])$ such that $sb \in \UniScheme(\Z[1/p])$. But then we necessarily have $sb.x \in \HyperOne \intersect \HyperTwo \intersect \EApartment = (\VecOne)^\perp \intersect (\VecTwo)^\perp$. Therefore, by the first part, there is a $t \in \DGroup \intersect \Torus$ such that $tsb.x \in K$. Since $tsb \in \DGroup$ this closes the proof.
\end{proof}

Since $\EBuilding$ is locally compact we get immediately:

\begin{cor}
\label{cor:cocompact}
For every $s > 0$ the action of $\DGroup$ on $(\BuseTwo)^{-1}([0,s]) \intersect \HyperOne$ is cocompact.
\end{cor}

The connectivity of horospheres in euclidean buildings has been established by Kai-Uwe Bux and Kevin Wortman \cite{buxwor11}:

\begin{thm}
\label{thm:horoball_connectivity}
Let $\Building$ be a thick euclidean building and $\AltPointAtInfinity \in \partial \Building$. Let $\Busemann$ be a Busemann function centered at $\AltPointAtInfinity$. Let $\Building_0$ be the least factor of $\Building$ such that $\AltPointAtInfinity \in \partial \Building_0$ and let $m$ be its dimension. Then for $r \le s$ the set $\Busemann^{-1}([r,s])$ is $(m-2)$-connected.
% Moreover the map
% \[
% \pi_{k-1}(\Busemann^{-1}([r,s])) \to \pi_{k-1}(\Busemann^{-1}([r,t]))
% \]
%is surjective for every $t \ge s$ but is not injective for some $t>s$.
Moreover there is a $t \ge s$ such that the map
\[
\pi_{k-1}(\Busemann^{-1}([r,s])) \to \pi_{k-1}(\Busemann^{-1}([r,t]))
\]
is not injective.
In particular $\Busemann^{-1}([r,s])$ is not $(m-1)$-connected.
\end{thm}

Since this is slightly stronger than \cite[Theorem~7.7]{buxwor11}, we briefly sketch how their machinery gives our statement.

\begin{proof}[Proof sketch.]
Let $\Busemann$ be a Busemann function centered at $\AltPointAtInfinity$. In general, if $X = X_0 \times X_1$ with $\AltPointAtInfinity \in X_0$, then $\Busemann$ is constant on $\{x\} \times X_1$ for every $x \in X_0$. That is, $\Busemann^{-1}([r,s]) = \Busemann|_{X_0}^{-1}([r,s]) \times X_1$. Since $X_1$ is contractible this shows in particular that $\Busemann^{-1}([r,s])$ and $\Busemann|_{X_0}^{-1}([r,s])$ are homotopy equivalent. So we may assume that $X_0 = X$.

Now the statement that $\Busemann^{-1}([r,s])$ is $(m-2)$-connected is almost \cite[Theorem~7.7]{buxwor11}, except that there it is stated for $\Busemann^{-1}((-\infty,s])$. But $\Busemann$ is a concave function, so horoballs are convex and $\Busemann^{-1}([r,s])$ is a deformation retract of $\Busemann^{-1}((-\infty,s])$.

It remains to verify that $\pi_{k-1}(\Busemann^{-1}([r,s])) \to \pi_{k-1}(\Busemann^{-1}([r,t]))$ is not injective for sufficiently large $t$. In the language of \cite{buxwor11} this requires showing that there is a barycenter $(\mathring{\tau_1},\ldots,\mathring{\tau_n})$ of $\Busemann$-height greater than $r$ so that $\Lk^\downarrow(\mathring{\tau_1},\ldots,\mathring{\tau_n})$ is not $(m-1)$-connected. But we can take each $\tau_i$ to be a special vertex of the corresponding factor. Then $\Lk^\downarrow(\mathring{\tau_i})$ is an open hemisphere complex in an irreducible, thick spherical building. These are not contractible by \cite[Theorem~B]{schulz10}, cf.\ the proof of Lemma~6.6 in \cite{buxwor11}. Therefore $\Lk^\downarrow(\mathring{\tau_1},\ldots,\mathring{\tau_n})$ is not $(m-2)$-connected which gives the desired statement.
\end{proof}

With these preparations in place it is a routine matter to prove the first part of the Main Theorem:

\begin{thm}
\label{thm:classical_finiteness_length}
The group $\DGroup = \AbelsScheme(\Z[1/p])$ is of type $\ClassicalF_{n-1}$ but not of type $\ClassicalFP_n$.
\end{thm}

\begin{proof}
By Lemma~\ref{lem:identification} the set $Z = (\BuseTwo)^{-1}([0,\infty)) \intersect \HyperOne$ is homeomorphic to a horoball in $\Building$ which is contractible being a convex subset of a $\CAT(0)$ space.

We want to apply Brown's criterion, Theorem~\ref{thm:browns_criterion}. The filtration we consider is
\[
Z_i \defeq (\BuseTwo)^{-1}([0,i]), i \in \N \text{ .}
\]
The action on each of these spaces is cocompact by Corollary~\ref{cor:cocompact}. The stabilizers are of type $\ClassicalF_\infty$ by \cite[Theorem~B(b)]{abebro87}. By Lemma~\ref{lem:identification} the terms of the filtration are homeomorphic to the intersection of a horoball and a horoball complement in $\Building$. Since $\Building$ is irreducible, Theorem~\ref{thm:horoball_connectivity} implies that they are $(n-2)$-connected, so in particular the system $(\pi_{k}(Z_i))_i$ is essentially trivial for $k<n-1$. The theorem also implies that the system $(\tilde{H}_{n-1}(Z_i))_i$ not essentially trivial.
\end{proof}

\section{Bredon-finiteness properties}
\label{sec:bredon}

To determine the Bredon-finiteness properties of $\DGroup$ we have to understand the torsion and its fixed point sets.

\begin{lem}
\label{lem:only_2-torsion}
Every torsion element of $\DGroup$ has order (at most) $2$. In fact the same is true of every torsion element of $\BorelScheme(\Z[1/p])$.
\end{lem}

\begin{proof}
Consider the homomorphism $\GroupRetraction|_{\BorelScheme(\Z[1/p])} \colon \BorelScheme(\Z[1/p]) \to \TorusScheme(\Z[1/p])$. Its kernel is $\UniScheme(\Z[1/p])$ which is torsion-free. The image $\TorusScheme(\Z[1/p]) \cong (\Z[1/p]^\times)^{n+1}$ is isomorphic to $(\{\pm 1\} \times \Z)^{n+1}$ and therefore contains only torsion of order $2$. Thus if $g \in \DGroup$ has finite order, then $\GroupRetraction(g^2) = 1$ and hence $g^2=1$.
\end{proof}

\begin{lem}
\label{lem:conjugate_to_diagonal}
Let $R$ be an integral domain in which $2$ is a unit. Every element of order $2$ in $\BorelScheme(R)$ is conjugate via an element of $\UniScheme(R)$ to a diagonal matrix.
\end{lem}

\begin{proof}
If $g^2 = \idmatrix$, then the diagonal entries $d_i$ of $g$ satisfy $d_i^2=1$. Since $R$ is an integral domain the only solutions of this equation are $\pm 1$. Hence every column of $g$ gives rise to a column of precisely one of
\[
\frac{1}{2}(\idmatrix + g) \quad \text{and} \quad \frac{1}{2}(\idmatrix -g)
\]
whose diagonal entry is $1$.
Collecting these columns gives a matrix in $\UniScheme(R)$ whose columns are eigenvectors of $g$.
\end{proof}

Putting both Lemmas together, we obtain:

\begin{prop}
\label{prop:conjugate_to_diagonal_group}
Every finite subgroup of $\DGroup$ is conjugate to a subgroup of $\DGroup \intersect \Torus$.
\end{prop}

\begin{proof}
Lemma~\ref{lem:only_2-torsion} implies that every finite subgroup of $\DGroup$ is abelian. So if $g$ and $h$ span a finite group, then $h$ leaves the eigenspaces of $g$ invariant. Applying Lemma~\ref{lem:conjugate_to_diagonal} inductively gives a a basis of $(\Z[1/p])^{n+1}$ of simultaneous eigenvectors of any finite subgroup.
\end{proof}

We can now explain the significance of the remaining notions from Section~\ref{sec:precise}. To understand torsion in $\DGroup$ it suffices by Proposition~\ref{prop:conjugate_to_diagonal_group} to understand torsion in $\Torus \intersect \DGroup$. Every diagonal involution can be described by a partition into the indices with entry $+1$ and $-1$ respectively. This partition is admissible if and only if the involution is an element of $\DGroup$.

The fixed point set of such an element will turn out to be the extended Bruhat--Tits building of its centralizer and in particular decomposes into a direct product of the extended buildings corresponding to the $+1$ and $-1$ eigenspace respectively. Clearly, the fixed point set $\FixedPointSet$ of a finite group is the intersection of the fixed point sets of its generators and therefore decomposes as a product of buildings that correspond to blocks of an admissible partition.

Each of these extended buildings is a direct product of a line and a building. If the least factor of $\FixedPointSet$ that contains $\PAIOT$ in its boundary has a line as a direct factor, then $\FixedPointSet \intersect \HyperOT$ is contractible. If on the other hand $\PAIOT$ is contained in the boundary of the building factors, which happens if and only if the partition is a partition of $\VecOT$, then Theorem~\ref{thm:horoball_connectivity} implies that the connectivity of $\FixedPointSet \intersect \HyperOT$ is determined by the least factor that contains it. A building factor contributes if and only if its block is essential and therefore the Bredon-finiteness length of $\DGroup$ is controlled by the essential dimension.

\medskip

Fixed point sets of finite order automorphisms of $\EBuilding$ are generally well-studied, see for example \cite{prayu01}. Inner involutions, that is automorphisms that come from involutions in $\GL_{n+1}(\LocalField)$ are particularly easy to understand. Their fixed point sets can be described as follows.

\begin{prop}
\label{prop:involution_fixed}
Let $\LocalField$ be a local field of residue characteristic $\ne 2$. Let $\Involution \in \GL_{n+1}(\LocalField)$ be an involution. Let $V^+$ and $V^-$ be the eigenspaces of $\Involution$ to the eigenvalue $+1$ and $-1$ respectively. The fixed point set of $\Involution$ on $\EBuilding$ is equivariantly isometric to the extended building associated to the group $\GL(V^+) \times \GL(V^-)$ (which is the centralizer of $\Involution$ in $\GL_{n+1}(\LocalField)$).
\end{prop}

We give two proofs that are essentially the same but refer to different models of $\EBuilding$. For the first recall that $\EBuilding$ is a simplicial complex whose vertices are $\Ring$-lattices in $\LocalField^{n+1}$, see Appendix~\ref{sec:extended_simplicial}. Here $\Ring$ is the valuation ring of $\LocalField$.

\begin{proof}[Proof~1]
Let $\Lattice$ be a lattice in $V = \LocalField^{n+1}$. Let $\Lattice^+ = \Lattice \intersect V^+$ and $\Lattice^- = \Lattice \intersect V^-$. We have to show that $\Lattice$ is $\Involution$-fix if and only if $\Lattice = \Lattice^+ + \Lattice^-$ because the lattices meeting the second condition are the ones lying in the building associated to $\GL(V^+) \times \GL(V^-)$.

Clearly if $\Lattice = \Lattice^+ + \Lattice^-$, then $\Involution \Lattice = \Lattice^+ - \Lattice^- = \Lattice$, so $\Lattice$ is $\Involution$-invariant.

For the converse note that $2 \in \Ring^\times$ by assumption. Assume that $\Lattice$ is $\Involution$-invariant and let $\LCVector \in \Lambda$. Then
\[
\LCVector = \frac{1}{2}(\LCVector + \Involution\LCVector) + \frac{1}{2}(\LCVector - \Involution\LCVector)
\]
where $1/2(\LCVector \pm \Involution\LCVector) \in \Lambda^\pm$. This closes the proof.
\end{proof}

For the second proof recall that the points of $\EBuilding$ correspond to splitable norms on $\LocalField^{n+1}$, see \cite[Théorème~2.11]{brutit84b}. Here norms are understood additively as in \cite[1.1]{brutit84b}: A norm on a $\LocalField$-vector space $V$ is a map $\LocalField \to \R \union \{\infty\}$ such that for $\LCVector, \LCVector' \in V$ and $k \in \LocalField$ the following hold:
\begin{enumerate}
\item $\Norm(k \LCVector) = \Valuation(k) + \Norm(\LCVector)$,
\item $\Norm(\LCVector + \LCVector') \ge \inf\{\Norm(\LCVector),\Norm(\LCVector')\}$, and
\item $\Norm(\LCVector) = \infty$ if and only if $\LCVector = 0$.
\end{enumerate}
A norm $\Norm$ is said to \emph{split} over a decomposition $V = V_1 \oplus V_2$ if $\Norm(\LCVector_1 + \LCVector_2) = \inf\{\Norm(\LCVector_1),\Norm(\LCVector_2)\}$ for $\LCVector_i \in V_i$, \cite[1.4]{brutit84b}. This clearly gives a notion of when a norm splits over a decomposition into more than two summands and a norm on $V$ is said to be \emph{splitable} if there is a decomposition of $V$ into one-dimensional subspaces over which it splits.

\begin{proof}[Proof~2]
It suffices to show that a norm $\Norm$ is $\Involution$-invariant if and only if splits over $V^+ \oplus V^-$.

Again it is clear that $\Norm$ if $\Involution$-invariant if it splits over $V^+ \oplus V^-$.

So suppose that $\Norm$ is $\Involution$-invariant and let $\LCVector \in \LocalField^{n+1}$. Write
\[
\LCVector = \frac{1}{2}(\LCVector + \Involution\LCVector) + \frac{1}{2}(\LCVector - \Involution\LCVector)
\]
where $1/2(\LCVector \pm \Involution\LCVector) \in V^\pm$.
Then $\Norm(\LCVector) \ge \inf \{1/2(\LCVector + \Involution\LCVector),1/2(\LCVector - \Involution\LCVector)\}$ but also
\[
\Norm\left(\frac{1}{2}(\LCVector \pm \Involution\LCVector)\right) \ge \inf\left\{\Norm(\frac{1}{2}\LCVector),\Norm(\frac{1}{2}\Involution\LCVector)\right\} = \Norm(\LCVector)
\]
because $\Norm$ is $\Involution$ invariant and $2 \in \Ring^\times$. This shows that $\Norm(\LCVector) = \inf\{1/2(\LCVector + \Involution\LCVector),1/2(\LCVector - \Involution\LCVector)\}$ as desired.
\end{proof}

For diagonal matrices we get the following more explicit statement. For brevity we write $\pm$ to mean either $+$ or $-$ consistently in each expression.

\begin{cor}
\label{cor:fixed_point_set_structure}
Assume that the residue characteristic of $\LocalField$ is not $2$. Let $\Involution = \diag(d_i)$ be of order $2$ in $\Torus$. Let $J^\pm = \{i \in \{1,\ldots,n+1\} \mid d_i = \pm 1\}$.

The fixed point set $\FixedPointSet$ of $\Involution$ decomposes as a direct product
$\FixedPointSet = \Building^+ \times \ExtendingFactor^+ \times \Building^- \times \ExtendingFactor^-$
where $\Building^\pm$ are buildings of type $\tilde{A}_{\abs{J^\pm} - 1}$ and $\ExtendingFactor^\pm$ are euclidean lines. More precisely, $\Retraction(\Building^\pm) = \gen{\ABasis_i - \ABasis_j \mid i,j \in I^\pm}$ and $\ExtendingFactor^\pm = \gen{\sum_{i \in J^\pm} \ABasis_i}$.
\end{cor}

\begin{proof}
Let $\Basis_1,\ldots,\Basis_{n+1}$ be the standard basis for $\LocalField^{n+1}$. The eigenspaces of $\Involution$ are $V^\pm = \gen{\Basis_i \mid i \in J^\pm}$. Let $\Building^{\pm 1}$ be the extended building of $\GL(V^\pm)$. Since the inclusion $\Building^{+1} \times \Building^{-1} \to \EBuilding$ is $\GL(V^+) \times \GL(V^-)$-equivariant we can determine the factors by looking at the invariant subspaces. Moreover, since everything commutes with $\Retraction$, it suffices to look at the action of $\Torus$ on $\EApartment$. We see that $\Building^{\pm 1} \intersect \EApartment$ is just the span of the $\ABasis_i, i \in J^\pm$ and that $X^{\pm 1}$ decomposes further as $\Building^\pm \times \ExtendingFactor^\pm$.
\end{proof}

As a consequence we get:

\begin{prop}
\label{prop:fixed_point_set_structure}
Let $F$ be a finite subgroup of $\DGroup \intersect \Torus$. Then there is an admissible partition of $I$ into blocks $J_1,\ldots,J_k$ such that the fixed point set $\FixedPointSet$ of $F$ decomposes as a direct product
\[
\FixedPointSet = X_1 \times L_1 \times \ldots \times X_k \times L_k
\]
with $L_\ell$ a euclidean line and $X_\ell$ a $(\abs{J_\ell}-1)$-dimensional building. The factors satisfy $\Retraction(X_\ell) = \gen{\ABasis_i - \ABasis_j \mid i \in J_\ell}$ and $L_\ell = \gen{\sum_{i \in I_\ell} \ABasis_i}$.

Conversely, if $\Partition$ is the common refinement of elementary admissible partitions $\Partition_1,\ldots,\Partition_r$ and $\Involution_1,\ldots,\Involution_r \in \DGroup$ are the corresponding involutions, then the partition arising above for $F=\gen{\Involution_1,\ldots,\Involution_r}$ is $\Partition$.
\end{prop}

We are now ready to prove the second part of the Main Theorem.

\begin{thm}
$\DGroup$ is of type $\BredonFPfin_{m-1}$ but not of type $\BredonFPfin_{m}$.
\end{thm}

\begin{proof}
We consider the same setup as in the proof of Theorem~\ref{thm:classical_finiteness_length} but this time apply Theorem~\ref{thm:browns_criterion_for_bredon} instead of Brown's classical criterion.

The stabilizers are of type $\BredonFPfin_\infty$ by \cite[Proof~of~Theorem B(b)]{abebro87} and \cite[Theorem~1.1]{kromapnuc09}.

We will verify the following statements, which are stronger than the needed connectivity hypotheses:
\begin{enumerate}
\item That $\FiltrationSet_i^F$ is $(m-2)$-connected for all $i \in \N$ and every finite $F \le \DGroup$.\label{item:sublevel_sets_connected}
\item And that there is a finite $F \le \DGroup$ such that the maps of the system
$(\tilde{H}_{m-1}(\FiltrationSet_i^F))_{i \in \N}$ are infinitely often not
injective.\label{item:sublevel_set_not_connected}
\end{enumerate}
From these assertions the result follows.

The case $F = 1$ of \eqref{item:sublevel_sets_connected} has already been verified in the proof of Theorem~\ref{thm:classical_finiteness_length}. So now we look at a nontrivial finite subgroup $F$.
%, which by Lemma~\ref{lem:only_2-torsion} is $2$-elementary abelian.
By Proposition~\ref{prop:conjugate_to_diagonal_group} $F$ is conjugate to a group of diagonal matrices and since conjugation does not change the homotopy type of the fixed point set, we may as well assume that $F$ is diagonal.

The fixed point set $\FixedPointSet = (\EBuilding)^F$ is described by Proposition~\ref{prop:fixed_point_set_structure} and decomposes according to an admissible partition $\Partition = \{J_1,\ldots,J_k\}$ as a product of euclidean buildings $\Building_\ell,1 \le \ell \le k$ and a euclidean space $L_1 \times \ldots \times L_k$. Proposition~\ref{prop:generalized_identification} implies that the map $\Project$ identifies the intersection $\FixedPointSet \intersect \HyperOne$ with $\FixedPointSet \intersect \Building$ in such a way that horoballs around $\PAITwo$ are mapped to horoballs around $\PAIOT$.

If $\PAIOT$ is not contained in the boundary of the product of the $\Building_\ell$, then it includes an acute angle with the endpoint of a direct factor of $\FixedPointSet$ that is a euclidean line. In that case Proposition~\ref{prop:generalized_identification} implies that $\FixedPointSet_i$ is contractible.
That $\PAIOT$ is contained in the boundary of the product of the $\Building_\ell$ is equivalent to the condition that $\VecOT$ is perpendicular to $L_\ell$ for all $\ell$, which is to say that $\Partition$ is a partition of $\VecOT$. If this is the case, then the minimal factor of $\FixedPointSet$ that contains $\PAIOT$ in its boundary is the product of those $\Building_\ell$ for which $J_\ell$ is essential. Therefore Theorem~\ref{thm:classical_finiteness_length} implies  that $\FixedPointSet \intersect \HyperOT$ is $(\ed(\Partition)-2)$-connected and in particular $(m-2)$-connected.

Finally we verify \eqref{item:sublevel_set_not_connected}. If $m=n$, the statement has been verified in the proof of Theorem~\ref{thm:classical_finiteness_length} for $F$ the trivial group. So assume $m<n$. Let $\Partition$ be an admissible partition of $\VecOT$ with $\ed(\Partition) = m$. Then $\Partition$ is the coarsest common refinement of essentially admissible partitions $\Partition_1, \ldots, \Partition_r$ that correspond to diagonal involutions $\Involution_1,\ldots,\Involution_r \in \DGroup$. Let $F=\gen{\Involution_1,\ldots,\Involution_r}$ and let $\FixedPointSet \defeq (\EBuilding)^F$ be its fixed point set. Proposition~\ref{prop:fixed_point_set_structure} describes the structure of $\FixedPointSet$. In particular it implies that $\PAIOT$ lies in the boundary of a factor of $\FixedPointSet$ that is a building of dimension $m$. Therefore Theorem~\ref{thm:horoball_connectivity} shows that the directed system $\tilde{H}_{m-1}(\FixedPointSet_i)$ is infinitely often not injective.
\end{proof}

\appendix

\section{The extended building of $\GL_{n}(\LocalField)$ as a simplicial complex}
\label{sec:extended_simplicial}

Let $\LocalField$ be a field equipped with a discrete valuation, let $\Ring$ be its valuation ring, and let $\Uniformizer$ be a uniformizing element. Let $V = \LocalField^{n}$. By an \emph{$\Ring$-lattice in $V$} (or just a lattice) we mean an $\Ring$-submodule $\Lambda$ of $V$ such that the map $\LocalField \otimes_{\Ring} \Lambda \to V$ is an isomorphism. We denote by $\ESimpBuilding$ the simplicial complex whose vertices are the $\Ring$-lattices in $V$ and whose simplices are flags
\[
\Lambda_0 \le \ldots \le \Lambda_k
\]
of lattices such that $\Uniformizer \Lambda_k \le \Lambda_0$.

Clearly $\GL(V)$ acts on $\ESimpBuilding$. Taking the quotient modulo the action of $\LocalField^\times$ gives a projection
\[
\delta \colon \ESimpBuilding \to \SimpBuilding
\]
where $\SimpBuilding$ has as vertices homothety classes $\EqClass\Lambda$ of lattices $\Lambda$. It is clear from the definition that $\SimpBuilding$ is just the affine building associated to $\SL_{n}(V)$, see for example \cite[Chapter~9.2]{ronan}. In particular, $\Building$ can be regarded as the geometric realization of $\SimpBuilding$.

\begin{figure}[htb]
\begin{center}
\includegraphics{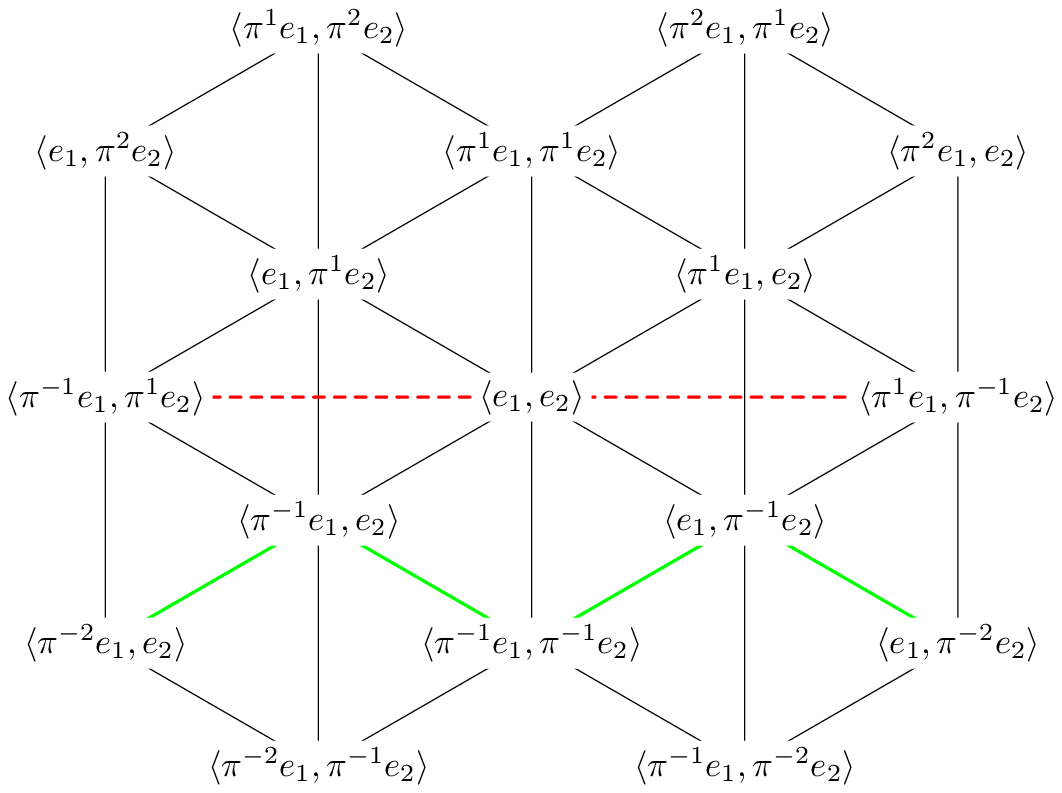}
\end{center}
\caption{Part of the fundamental apartment of the extended building $\ESimpBuilding$ for $\SL_2(\LocalField)$. A simplicial subcomplex that is isomorphic to the fundamental apartment of $\SimpBuilding$ is drawn in bold green. A subspace that is isometric to a fundamental apartment of $\Building$ is dashed in bold red.}
\label{fig:splittings}
\end{figure}

We want to see that similarly $\EBuilding$ can be regarded as the geometric realization of $\ESimpBuilding$. We have to be a little careful because even though the projection $\EBuilding \to \Building$ of metric spaces as well as the projection $\ESimpBuilding \to \SimpBuilding$ of simplicial complexes admit splittings (by isometric respectively simplicial embeddings) these splittings do not coincide under the identification we want to make (see Figure~\ref{fig:splittings}).

One way to deal with this would be to construct appropriate subdivisions of $\SimpBuilding$ and $\ESimpBuilding$ under which the metric splitting becomes simplicial. Instead, we exhibit an equivariant homeomorphism $\trealize{\ESimpBuilding} \to \trealize{\SimpBuilding} \times \SimpLine$ that is not induced by a simplicial map.

To do so we use the following definitions from \cite{grayson82}. The \emph{length} $\length(M)$ of a $\Ring$-module $M$ is the length of a maximal chain of proper submodules of $M$. For two arbitrary lattices $\Lambda$ and $\Lambda_0$ the \emph{index} $\ind(\Lambda,\Lambda_0)$ is $\length(\Lambda/\Lambda_1) - \length(\Lambda_0/\Lambda_1)$ for any lattice $\Lambda_1$ contained in $\Lambda$ and $\Lambda_0$. Finally we fix a lattice $\Lambda_0$ and define the map $\varepsilon \colon \trealize{\ESimpBuilding} \to \SimpLine$ by
\[
\varepsilon(\Lambda) = \frac{1}{n}\ind(\Lambda,\Lambda_0)
\]
on vertices and extending affinely. This last definition is different from Grayson's who wants the map to be simplicial.

\begin{prop}
The map $\trealize{\delta} \times \varepsilon \colon \trealize{\ESimpBuilding} \to \trealize{\SimpBuilding} \times \SimpLine$ is a $\GL(V)$-equivariant homeomorphism.
\end{prop}

\begin{proof}
That the map is continuous and $\GL(V)$-equivariant is clear.

Bijectivity follows from Lemma~\ref{lem:convex_combinations} below: let $\Point \in \trealize{\SimpBuilding}$ be the convex combination
\[
\Point = \sum_{j=0}^{n-1} \alpha_j \EqClass{\Lambda_j}
\]
where we choose the representatives $\Lambda_j$ so that $\varepsilon(\Lambda_j) = j/n$. Let further $r \in \R$. The lemma exhibits an $i \in \Z$ and a $\beta \in [0,1)$ such that
\begin{align*}
\tilde{\Point} =  &\sum_{j=1}^{n-1} \alpha_{(i+j) \omod n} (\Uniformizer^{(i+j) \odiv n} \Lambda_{(i+j) \omod n})\\
& + \beta \alpha_{i \omod n} (\Uniformizer^{i \odiv n} \Lambda_{i \omod n}) + (1 - \beta) \alpha_{i \omod n} (\Uniformizer^{(i + n)\odiv n} \Lambda_{i \omod n})
\end{align*}
satisfies $\varepsilon(\tilde{\Point}) = r$, $\delta(\tilde{\Point}) = \Point$, and $\tilde{\Point}$ is unique with these properties.

To see that the map is closed, it suffices to consider sets of the form $C = \trealize{\sigma} \times [a,b]$ for $\sigma$ a cell of $\trealize{\SimpBuilding}$ and show that the restriction
\[
(\trealize{\delta} \times \varepsilon)^{-1}(C) \to C
\]
is closed. But sets of the form $(\trealize{\delta} \times \varepsilon)^{-1}(C)$ are clearly compact.
\end{proof}

\begin{lem}
\label{lem:convex_combinations}
Let $\alpha_j,0\le j \le n-1$ be such that $\alpha_j \ge 0$ and $\sum \alpha_j = 1$. For $i \in \Z$ set
\[
c_i = \sum_{j=0}^{n-1} \alpha_{(i+j) \omod n} \frac{i+j}{n}\text{ .}
\]
The intervals $[c_i,c_i+\alpha_{i \omod n})$ with $i \in \Z$ and $\alpha_{i \omod n} > 0$ disjointly cover $\R$.

In other words, for every $r \in \R$ there are $i \in \Z$ and $\beta \in [0,1)$ so that
\[
r = \beta\alpha_{i \omod n} \frac{i}{n} + \sum_{1}^{n-1} \alpha_{(i+j) \omod n} \frac{i+j}{n} + (1-\beta)\alpha_{(i+n) \omod n} \frac{i+n}{n}
\]
and these are unique if we require $\alpha_{i \omod n} \ne 0$.
\end{lem}

\begin{proof}
This amounts to saying $c_{i+1} - c_i = \alpha_{i \omod n}$, which is elementary.
\end{proof}

\providecommand{\bysame}{\leavevmode\hbox to3em{\hrulefill}\thinspace}
\providecommand{\MR}{\relax\ifhmode\unskip\space\fi MR }
\providecommand{\MRhref}[2]{%
  \href{http://www.ams.org/mathscinet-getitem?mr=#1}{#2}
}
\providecommand{\href}[2]{#2}

\bigskip

\noindent
\begin{minipage}[t]{0.5\textwidth}
    Stefan Witzel\\
    Mathematisches Institut\\
    Universität Münster\\
    Einsteinstraße 62\\
    48149 Münster\\
    Germany\\
    e-mail: \texttt{s.witzel@uni-muenster.de}
\end{minipage}


\begin{thebibliography}{KMPN11}

\bibitem[AB87]{abebro87}
Herbert Abels and Kenneth~S. Brown, \emph{Finiteness properties of solvable
  {$S$}-arithmetic groups: an example}, J. Pure Appl. Algebra \textbf{44}
  (1987), 77--83.

\bibitem[AB08]{abrbro}
Peter Abramenko and Kenneth~S. Brown, \emph{Buildings: Theory and
  applications}, Graduate Texts in Mathematics, vol. 248, Springer, 2008.

\bibitem[Abe79]{abels79}
Herbert Abels, \emph{An example of a finitely presented solvable group},
  Homological group theory ({P}roc. {S}ympos., {D}urham, 1977), London Math.
  Soc. Lecture Note Ser., vol.~36, Cambridge Univ. Press, 1979, pp.~205--211.

\bibitem[Bre67]{bredon67}
Glen~E. Bredon, \emph{Equivariant cohomology theories}, Lecture Notes in
  Mathematics, No. 34, Springer-Verlag, 1967.

\bibitem[Bro84]{brown84}
Kenneth~S. Brown, \emph{Presentations for groups acting on simply-connected
  complexes}, J. Pure Appl. Algebra \textbf{32} (1984), no.~1, 1--10.

\bibitem[Bro87]{brown87}
\bysame, \emph{Finiteness properties of groups}, J.\ Pure Appl.\ Algebra
  \textbf{44} (1987), 45--75.

\bibitem[BT84]{brutit84b}
François Bruhat and Jacques Tits, \emph{Sch\'emas en groupes et immeubles des
  groupes classiques sur un corps local}, Bull. Soc. Math. France \textbf{112}
  (1984), no.~2, 259--301.

\bibitem[BW11]{buxwor11}
Kai-Uwe Bux and Kevin Wortman, \emph{Connectivity properties of horospheres in
  {E}uclidean buildings and applications to finiteness properties of discrete
  groups}, Invent.\ Math. \textbf{185} (2011), 395--419.

\bibitem[FW12]{fluwit}
Martin Fluch and Stefan Witzel, \emph{Brown's criterion in {B}redon homology},
  arXiv:1206.0962, Preprint, 2012.

\bibitem[Gra82]{grayson82}
Daniel~R. Grayson, \emph{Finite generation of {$K$}-groups of a curve over a
  finite field (after {D}aniel {Q}uillen)}, Algebraic {$K$}-theory, {P}art {I}
  ({O}berwolfach, 1980), Lecture Notes in Math., vol. 966, Springer, 1982,
  pp.~69--90.

\bibitem[KMPN09]{kromapnuc09}
Peter~H. Kropholler, Conchita Martinez-P{\'e}rez, and Brita E.~A. Nucinkis,
  \emph{Cohomological finiteness conditions for elementary amenable groups}, J.
  Reine Angew. Math. \textbf{637} (2009), 49--62.

\bibitem[KMPN11]{kocmarnuc11}
Dessislava~H. Kochloukova, Conchita Mart{\'{\i}}nez-P{\'e}rez, and Brita E.~A.
  Nucinkis, \emph{Centralisers of finite subgroups in soluble groups of type
  {${\rm FP}_n$}}, Forum Math. \textbf{23} (2011), no.~1, 99--115.

\bibitem[LN03]{leanuc03}
Ian~J. Leary and Brita E.~A. Nucinkis, \emph{Some groups of type {$VF$}},
  Invent. Math. \textbf{151} (2003), no.~1, 135--165.

\bibitem[L{\"u}c89]{lueck89}
Wolfgang L{\"u}ck, \emph{Transformation groups and algebraic {$K$}-theory},
  Lecture Notes in Mathematics, vol. 1408, Springer, Berlin, 1989.

\bibitem[PY02]{prayu01}
Gopal Prasad and Jiu-Kang Yu, \emph{On finite group actions on reductive groups
  and buildings}, Invent. Math. \textbf{147} (2002), 545--560.

\bibitem[Ron89]{ronan}
Mark Ronan, \emph{Lectures on buildings}, Perspectives in Mathematics, vol.~7,
  Academic Press, 1989.

\bibitem[Sch10]{schulz10}
Bernd Schulz, \emph{Spherical subcomplexes of spherical buildings},
  arXiv:1007.2407v4, Preprint, 2010.

\bibitem[Str84]{strebel84}
Ralph Strebel, \emph{Finitely presented soluble groups}, Group theory, Academic
  Press, London, 1984, pp.~257--314.

\end{thebibliography}
\end{document}